\documentclass[12pt]{amsart}

\usepackage{amssymb,latexsym}
\usepackage[enableskew]{youngtab}
\usepackage{verbatim,euscript}
\usepackage{fullpage}
\usepackage[all]{xy}

\usepackage[colorlinks=true,linkcolor=blue,urlcolor=violet,citecolor=magenta]{hyperref}

\input {cyracc.def}
\numberwithin{equation}{section}

\font\tencyr=wncyr10 
\font\tencyi=wncyi10 
\font\tencysc=wncysc10 
\def\rus{\tencyr\cyracc}
\def\rusi{\tencyi\cyracc}
\def\rusc{\tencysc\cyracc}

\newtheorem{thm}{Theorem}[section]
\newtheorem{lm}[thm]{Lemma}
\newtheorem{prop}[thm]{Proposition}
\newtheorem{conj}[thm]{Conjecture}

\theoremstyle{remark}

\newtheorem*{rema}{Remark}

\theoremstyle{definition}
\newtheorem{ex}[thm]{Example}
\newtheorem{df}{Definition}
\newtheorem{rmk}[thm]{Remark}

\newenvironment{proof*}
{\noindent {\sl Proof.}\quad }{\hfill $\square$}

\newcommand {\ah}{{\mathfrak a}}
\newcommand {\be}{{\mathfrak b}}

\newcommand {\g}{{\mathfrak g}}

\newcommand {\fH}{{\eus H}}

\newcommand {\n}{{\mathfrak n}}

\newcommand {\p}{{\mathfrak p}}

\newcommand {\te}{{\mathfrak t}}
\newcommand {\ut}{{\mathfrak u}}

\newcommand {\gH}{{\eus H}}

\newcommand {\esi}{\varepsilon}
\newcommand {\ap}{\alpha}

\newcommand {\lb}{\lambda}

\newcommand {\ca}{{\mathcal A}}

\newcommand {\BR}{{\mathbb R}}
\newcommand {\BC}{{\mathbb C}}
\newcommand {\BZ}{{\mathbb Z}}

\newcommand {\ess}{{\mathsf{Ess}}}

\newcommand {\hot}{{\mathsf{ht}}}

\newcommand {\GR}[2]{{\textrm{{\bf #1}}}_{#2}}

\newcommand {\un}{\underline}

\newcommand {\Ab}{\mathfrak{Ab}}
\newcommand {\Abo}{\overset{o}{\mathfrak{Ab}}}

\newcommand {\beq}{\begin{equation}}
\newcommand {\eeq}{\end{equation}}

\newcommand{\curge}{\succcurlyeq}
\newcommand{\curle}{\preccurlyeq}
\renewcommand{\le}{\leqslant}
\renewcommand{\ge}{\geqslant}

\newcommand{\ihp}{I(\hat\ap)}
\newcommand{\eus}{\EuScript}
\newcommand{\MS}{\textsf{MICS}}

\begin{document}
\setlength{\parskip}{2pt plus 4pt minus 0pt}
\hfill {\scriptsize May 4, 2015} 
\vskip1.5ex

\title[MICS and abelian ideals]{Minimal inversion complete sets and maximal abelian ideals}
\author{Dmitri I. Panyushev}
\address[]{
Institute for Information Transmission Problems of the R.A.S, B. Karetnyi per. 19, 
127051 Moscow,   Russia}
\email{panyushev@iitp.ru}
\begin{abstract}

Let $\g$ be a simple Lie algebra, $\be$ a fixed Borel subalgebra, and $W$ the Weyl group 
of $\g$. 
In this note, we study a relationship  between the maximal abelian ideals of $\be$ and the minimal inversion complete sets of $W$. The latter have been recently defined by Malvenuto et al. ({\it J.\,Algebra},
{\bf 424}\,(2015), 330--356.) 
\end{abstract}
\maketitle

\section*{Introduction}

\noindent
Recently, Malvenuto et al. \cite{MMOP} introduced the {\sl minimal inversion complete sets} (\MS\ for short) 
in finite Coxeter groups and determined the maximum cardinality of a \MS\ for the classical series 
of Weyl groups $W$ and $\GR{G}{2}$. They also gave a lower bound on the maximum cardinality of \MS\ for
the other exceptional Weyl groups. It is noticed in \cite{MMOP} that in the simply-laced case this 
maximum cardinality is related to the maximum dimension of abelian ideals of a Borel subalgebra $\be$ of a 
simple Lie algebra $\g$ with Weyl group $W$. In this article, we elaborate on the relationship between 
maximal abelian ideals of $\be$ and \MS\ in $W$. In the simply-laced case, we give a uniform construction of a \MS\ for {\bf any} maximal abelian ideal of $\be$. We also determine the essential set of roots for some \MS\ obtained. (See definitions below.)

Let $\Delta$ be an irreducible crystallographic root system in a real Euclidean vector space $V$ and 
$W\subset GL(V)$ the corresponding finite reflection group. Let $\Delta^+$ be a set of positive roots, 
$\Pi$ the set of simple roots in $\Delta^+$, and $\theta$ the highest root. We regard $\Delta^+$ as a poset
with respect to the usual order ``$\curge$''. If $w\in W$, then 
$\eus N(w)=\{\gamma\in\Delta^+\mid w(\gamma)\in -\Delta^+\}$ is the {\it inversion set\/} of $w$.

\begin{df}[\cite{MMOP}]      \label{def1}
A subset $\eus F$ of $W$ is a  {\it minimal inversion complete set\/} (=\,\MS), if
$\bigcup_{w\in\eus F} \eus N(w)=\Delta^+$ and 
the equality fails for any proper subset of $\eus F$. 
\end{df}
\begin{df}[\cite{MMOP}]  \label{def:essential}
A root $\gamma\in\Delta^+$ is said to be {\it essential\/} for a given \MS\ $\eus F$, if there is a unique
$w\in \eus F$ such that $\gamma\in\eus N(w)$. 
\end{df}
\noindent
Write $\ess(\eus F)$ for the set of all essential roots.
By Definition~\ref{def1}, each 
$\eus N(w)$, $w\in\eus F$, contains at least one essential root. 
Hence $\#\ess(\eus F)\ge \# \eus F$.
Picking just one essential root in every $\eus N(w)$ yields a subset of $\ess(\eus F)$ that plays an
important role in \cite{MMOP}. However,  
the whole set $\ess(\eus F)$ is of interest, too. We introduce the {\it defect\/} of  $\eus F$ as
$
    \mathrm{defect}(\eus F):=\#\ess(\eus F)-\#\eus F 
$
and consider it as another measure of ``goodness'' of $\eus F$. That is, for us, the best \MS\ are those with large
cardinality or small defect.
If $w_0\in W$ is the longest element, then 
$\eus F=\{w_0\}$ is a \MS\ and  $\ess(\{w_0\})=\Delta^+$. Hence this \MS\ is very bad, since
the cardinality is small and defect is large. 

Let $\g$ be a simple Lie algebra over $\BC$ with a fixed triangular decomposition
$\g=\ut^-\oplus\te\oplus\ut^+$. Here $\be=\te\oplus\ut^+$ is a fixed Borel subalgebra, $\Delta$ is the set of
$\te$-roots in $\g$, and $\Delta^+$ is the subset corresponding to $\ut^+$.
It is well known and easily seen that if $\ah$ is an abelian ideal of $\be$, then $\ah\subset\ut^+$ and hence
$\ah=\bigoplus_{\gamma\in I(\ah)}\g_\gamma$ for certain subset $I(\ah)\subset\Delta^+$. 
The subsets of the form $I(\ah)$ are characterised by the following two properties:

\textbullet \quad If $\gamma\in I(\ah)$, $\nu\in\Delta^+$, and $\gamma+\nu\in\Delta$, then
$\gamma+\nu\in I(\ah)$;  \ [$\ah$ is $\be$-stable]

\textbullet \quad If $\gamma,\gamma'\in I(\ah)$, then $\gamma+\gamma'\not\in \Delta$. \qquad
[$\ah$ is abelian]
 
\begin{thm}   \label{thm:intro1}
If all the roots of $\Delta$ have the same length (the {\sf simply-laced} or {\sf ADE} case), then there is a 
natural \MS\ associated with {\bf any} maximal abelian ideal $\ah$ of\/ $\be$. 
For any $\gamma\in I(\ah)$, we define the canonical
element $\tilde w_\gamma\in W$ such that 
$\eus F_\ah=\{\tilde w_\gamma\mid \gamma\in I(\ah)\}$ is a \MS. In particular, 
$\#(\eus F_\ah)=\dim\ah=\#I(\ah)$. Furtermore, $\ess(\eus F_\ah)\supset I(\ah)$.
\end{thm}

We provide a full description of $\ess(\eus F_\ah)$ for two classes of maximal abelian ideals $\ah$.
Let $\p \supset \be$ be a parabolic subalgebra with abelian nilradical. Then $\p=\p_\ap$ is a maximal parabolic subalgebra that  is determined by one simple root $\ap$, and its nilradical  $\n_\ap$ is a maximal abelian ideal of $\be$.  

\begin{thm}   \label{thm:intro2} 
Suppose that $\Delta$ is simply-laced and\/ $\n_\ap$ is an abelian nilradical such that $(\ap,\theta)=0$. 
Then\/ $\ess(\eus F_{\n_\ap})=I(\n_\ap)$. 
\end{thm}

\noindent If $\Delta$ is of type $\GR{A}{n}$, then $(\ap,\theta)\ne 0$ for some abelian nilradicals and  $\ess(\eus F_{\n_\ap})=\Delta^+$ in 
those cases, see Example~\ref{ex:An}.

In general, the maximal abelian ideals are naturally parameterised by the long simple 
roots~\cite[Corollary\,3.8]{imrn}. (We say more on this correspondence in Section~\ref{sect:odin}.)
In particular, in the {\sf ADE} case, there is a bijection between $\Pi$ and the maximal abelian ideals.
Suppose that $\theta$ is a fundamental weight (in the simply-laced case, this means that $\Delta$ is of type
$\GR{D}{n}$ or $\GR{E}{n}$). Write $\hat\ap$ for the unique simple root such that $(\theta,\hat\ap)\ne 0$.
The corresponding maximal abelian ideal $\hat\ah:=\ah_{\hat\ap}$ has the property that $I(\hat\ah)\subset
\gH:=\{\gamma\in \Delta^+\mid (\gamma,\theta) > 0\}$~\cite{imrn}.

\begin{thm}   \label{thm:intro3}
For the above maximal abelian ideal $\hat\ah$, we have $\ess(\eus F_{\hat\ah})=\gH$.
\end{thm}

\noindent 
Explicit computations for $\GR{D}{n}$ and $\GR{E}{n}$ ($n\le 6$) suggest that it might be true that if 
$\theta$ is fundamental, then 
$\ess(\eus F_{\ah})\subset 
I(\ah)\cup \gH$. 

The general theory of abelian ideals (of $\be$) is based on a relationship with the so-called {\it minuscule 
elements\/} of the affine Weyl group $\widehat W$ (the Kostant-Peterson theory, see \cite{ko98} and also 
\cite{cp1}). Proofs in this article heavily rely on some further results obtained in \cite{imrn,jems}. For instance, the construction of $\eus F_\ah$ (in the {\sf ADE} case!) exploits 
the simple root $\ap$ corresponding to the maximal abelian ideal $\ah=\ah_\ap$. Furthermore, if $\gamma\in I(\ah_\ap)$, then
$\gamma\curge \ap$ \cite[Theorem\,3.5]{jems}, and since $\|\gamma\|=\|\ap\|$, we are able to introduce 
the element of minimal length in $W$ taking $\gamma$ to $\ap$, which is denoted by $w_{\gamma,\ap}$.
This is the crucial step in constructing the elements $\tilde w_\gamma$ occurring in Theorem~\ref{thm:intro1}. Actually, the existence of the elements of minimal length $w_{\gamma,\mu}$ is proved for any pair of positive roots such that $\|\gamma\|=\|\mu\|$ and $\gamma\curge\mu$, see 
Prop.~\ref{prop:unique-short} and Remark~\ref{rem:passage-dual}(1).
 
The article is organised as follows. In Section~\ref{sect:odin}, we provide preliminaries on abelian ideals and
elements $w_{\gamma,\mu}\in W$ associated with a pair of positive roots such that $\|\gamma\|=\|\mu\|$ 
and $\gamma\curge\mu$. Theorem~\ref{thm:intro1} is proved in Section~\ref{sect:dva}. Section~\ref{sect:some-prop-roots} contains some preparatory properties of long roots that are needed in Section~\ref{sect:essential}, where we prove Theorems~\ref{thm:intro2} and \ref{thm:intro3}. In Section~\ref{sect:fin},
we discuss some conjectures on the essential set of $\eus F_\ap$ in the cases that are not covered by
Theorems~\ref{thm:intro2} and \ref{thm:intro3}.

{\small
{\bf Acknowledgements.} A part of this work was done during my visit to 
the Max-Planck-Institut f\"ur Mathematik (Bonn) in Spring 2015. 
}

\section{Maximal abelian ideals and simple roots}
\label{sect:odin}

\noindent
In what follows, we identify the abelian ideals of $\be$ with the corresponding sets of roots. Consequently, in 
place of the dimension of $\ah\subset\ut^+$, we deal with the cardinality of $I(\ah)\subset\Delta^+$, etc.
It is proved in \cite{imrn} that there is a one-to-one correspondence between the maximal abelian ideals and 
the long simple roots in $\Delta^+$. As our subsequent results on \MS\ heavily rely on that 
correspondence, we recall the necessary setup.

\un{Some Notation}. We refer to \cite{bour},\cite{hump} for basic results on root systems and Weyl groups. 
Write $(\ ,\ )$ for the $W$-invariant scalar product in $V$ and
$\Delta^+_l$ (resp. $\Delta^+_s$) for the set of long (resp. short) positive roots. In the {\sf ADE} case, all roots are assumed to be both long and short.
\\  \indent
-- \ $\gamma^\vee=2\gamma/(\gamma,\gamma)$ and $\sigma_\gamma\in W$ is the reflection with respect to $\gamma\in\Delta$.
If $\Pi=\{\ap_1,\dots,\ap_n\}$, then we also write $\sigma_i$ in place of $\sigma_{\ap_i}$.
\\  \indent
-- \ $\rho=\frac{1}{2}\sum_{\gamma\in\Delta^+}\gamma$, \ 
$\rho^\vee=\frac{1}{2}\sum_{\gamma\in\Delta^+}\gamma^\vee$, \  
 and $\theta$ is the highest root in $\Delta^+$;
\\  \indent
-- \ $\ell(\cdot)$ is the length function on $W$ with respect to $\{\sigma_\ap\mid \ap\in\Pi\}$;
\\ \indent
-- \ $\fH=\{\gamma\in\Delta^+\mid (\gamma,\theta)\ne 0\}=\{\theta\}\cup
\{\gamma\in\Delta\mid (\gamma,\theta^\vee)=1\}$;
\\ \indent
-- \ If $\gamma=\sum_{\ap\in\Pi} c_\ap \ap$, then $\hot(\gamma)=\sum_{\ap\in\Pi} c_\ap$.
\\ \indent
-- \ The Coxeter number of $\Delta$ is $h=h(\Delta):=\hot(\theta)+1$.
\subsection{}  \label{subs:1.1} 
Let $\Ab$ (resp. $\Abo$) denote the set of all (resp. nonzero) abelian ideals of $\be$.
We regard $\Ab$ as a poset with respect to inclusion.

{\sf 1$^o$.}  There is a natural surjective map $\tau:\Abo\to \Delta^+_l$, and each fibre
$\tau^{-1}(\mu)=:\Ab_\mu$ ($\mu\in\Delta^+_l$) contains a unique maximal and a unique minimal 
ideal~\cite[Theorem\,3.1]{imrn}.
Write $I(\mu)_{max}$ and $I(\mu)_{min}$ for these extreme elements of $\Ab_\mu$. Say that
$I(\mu)_{min}$ (resp. $I(\mu)_{max}$) is the $\mu$-{\it minimal} (resp. $\mu$-{\it maximal}) {\it ideal}.

{\sf 2$^o$.}   The $\mu$-{minimal} ideals admit the following characterisation:

{\it For $I\in \Abo$, we have \ $I=I(\mu)_{min}$ for some $\mu\in\Delta^+_l$ {\sl if and only if\/} $I\subset\fH$~\cite[Theorem\,4.3]{imrn}.}
\\
Furthermore, all other ideals $I\in\Ab_\mu$ have the property that 
$I\cap \fH=I(\mu)_{min}$~\cite[Prop.\,3.2]{jems}. In particular,
$I(\mu)_{max}$ is maximal among all abelian ideals having the prescribed intersection with $\fH$. It is also 
known that $I(\mu)_{min}=I(\mu)_{max}$ if and only if $\mu\in\gH$ (i.e., $(\mu,\theta)\ne 0$), see
\cite[Theorem\,5.1]{imrn}.

{\sf 3$^o$.}  For any $\mu\in\Delta^+_l$, $W$ contains a unique element of minimal length taking 
$\theta$ to $\mu$~\cite[Theorem\,4.1]{imrn}. This element is denoted by $w_\mu$ in \cite{imrn} and here
we write $w_{\theta,\mu}$ for it. Its inverse has the
following description:
\beq   \label{eq:w-inverse}
    \eus N(w_{\theta,\mu}^{\ -1})=\{\nu\in\Delta^+\mid (\nu,\mu^\vee)=-1\} .
\eeq
The $\mu$-minimal ideal can be constructed using $w_{\theta,\mu}$, see \cite[Theorem\,4.2]{imrn}.
 In particular,  $\#I(\mu)_{min}=\ell(w_{\theta,\mu})+1=(\rho,\theta^\vee-\mu^\vee)+1$. See also 
 Lemma~\ref{lm:mu-min} below.

{\sf 4$^o$.}  If $\mu=\ap\in \Delta^+_l\cap \Pi=:\Pi_l$, then $I(\ap)_{max}$ is not only the maximal element of
$\Ab_\ap$, but  is also a (globally) maximal abelian ideal, and this yields the above-mentioned bijection. 
The following is a useful complement to the theory of $\mu$-minimal ideals.

\begin{lm}   \label{lm:mu-min}
For any $\mu\in\Delta^+_l$, we have (1) \ $\eus N(w_{\theta,\mu})\subset \gH\setminus\{\theta\}$ and
\[
   (2)\quad  I(\mu)_{min}=\{\theta\}\cup\{ \theta-\gamma\mid \gamma\in \eus N(w_{\theta,\mu})\} .
\]
\end{lm}
\begin{proof}
(1) \ If $w_{\theta,\mu}(\gamma)=-\nu\in -\Delta^+$, then $\nu\in \eus N(w_{\theta,\mu}^{\ -1})$. Therefore 
$(\nu,\mu^\vee)=-1$ and hence $(-\gamma,\theta^\vee)=-1$. Thus, 
$\gamma\in \gH\setminus\{\theta\}$.
\\ \indent
(2) \ The argument requires an explicit use of the affine root system and affine Weyl group. We give a brief sketch of notation referring to \cite[1.1]{imrn} for a full account. 
\\ \indent 
Let $\widehat\Delta=\{\Delta+k\delta\mid k\in\BZ\}$ be the affine root system and 
$\widehat\Pi=\Pi\cup \{\delta-\theta\}$ the set of affine simple roots. Here $\widehat\Delta$ is a subset of 
the vector space $\widehat V=V\oplus\BR\delta\oplus\BR\lb$, where $(V,\lb)=(V,\delta)=(\lb,\delta)=0$ and 
$(\delta,\lb)=1$. Let $\sigma_0$ denote the reflection (in $\widehat V$) with respect to $\ap_0=\delta-\theta$. 
The affine Weyl group $\widehat W$ is a subgroup of $GL(\widehat V)$ generated by 
$\sigma_\ap,\ (\ap\in\Pi)$ and $\sigma_0$.  Then $w_{\theta,\mu}\sigma_0\in \widehat W$ is minuscule
and its inversion set determines the abelian ideal $I(\mu)_{min}$~\cite[Theorem\,4.2]{imrn}. Namely,
\beq   \label{eq:minuscule}
   \text{ $\gamma\in I(\mu)_{min}$ \ if and only if \ $\delta-\gamma\in \eus N(w_{\theta,\mu}\sigma_0)$ }
\eeq
We have $\eus N(w_{\theta,\mu}\sigma_0)=\{\ap_0\}\cup \sigma_0(\eus N(w_{\theta,\mu}))$ and 
$\eus N(w_{\theta,\mu})\subset \gH\setminus \{\theta\}$. Therefore \\
$\sigma_0(\nu)=\nu-(\nu,\ap_0^\vee)=\nu+\ap_0=\delta-(\theta-\nu)$ for $\nu\in \eus N(w_{\theta,\mu})$.
Thus, 
\[
  \eus N(w_{\theta,\mu}\sigma_0)=\{\delta-\theta\}\cup \{\delta-(\theta-\nu)\mid 
  \nu\in \eus N(w_{\theta,\mu}) \}
\]
and comparing with Eq.~\eqref{eq:minuscule}, we conclude the proof.
\end{proof}

\subsection{} We regard $\Delta^+$ as a poset with respect to the usual order ``$\curge$''. This means that
$\gamma\curge\mu$ if and only if $\gamma-\mu$ is a linear combination of simple roots with {\sl nonnegative\/}  coefficients. If $\Xi\subset\Delta^+$, then $\min(\Xi)$ (resp. $\max(\Xi)$) is the set of minimal
(resp. maximal) elements of $\Xi$ with respect to  ``$\curge$''.

Let $\Pi_l$ (resp. $\Pi_s$) denote the set of long (resp. short) simple roots. We also need the ratio
$r=\|{\rm long}\|^2/\|{\rm short}\|^2$. In the simply-laced case $\Pi=\Pi_l=\Pi_s$ and $r=1$.
If $\gamma=\sum_{\ap\in\Pi} c_\ap\ap\in \Delta^+_l$, then $r$ divides $c_\ap$ whenever $\ap\in\Pi_s$ and
\beq   \label{eq:rho-gamma}
   (\rho,\gamma^\vee)=\sum_{\ap\in\Pi_l}c_\ap+\frac{1}{r}\sum_{\ap\in\Pi_s}c_\ap .
\eeq

\begin{prop}   \label{prop:unique-short}
Suppose that $\gamma,\mu\in\Delta^+_l$ and  $\gamma\curge\mu$.  Then there is a unique 
$w=w_{\gamma,\mu}\in W$ of minimal length such that $w_{\gamma,\mu}(\gamma)=\mu$. In this case, 
$\ell(w_{\gamma,\mu})=(\rho,\gamma^\vee)-(\rho,\mu^\vee)$.
\end{prop}
\begin{proof}
{\sf (1)} \ If $\gamma=\theta$ is the highest root, which is always long, then the required properties of
$w_{\theta,\mu}$ are proved in \cite[Section\,4]{imrn}. Therefore, we can use below  $w_{\theta,\gamma}$ and $w_{\theta,\mu}$.

{\sf (2)} \ Since $\|\gamma\|=\|\mu\|$, there is $w\in W$ such that $w(\gamma)=\mu$. For $\ap\in\Pi$, one 
computes that 
$(\rho, (\sigma_\ap(\gamma))^\vee)-(\rho,\gamma^\vee)=(\ap,\gamma^\vee)$. 
Since $(\ap,\gamma^\vee)\le 1$ if $\gamma\ne \ap$, one needs at least $(\rho,\gamma^\vee-\mu^\vee)$ steps (simple reflections) in order to reach
$\mu$ from $\gamma$.
That is, $\ell(w)\ge (\rho,\gamma^\vee-\mu^\vee)$.  

Arguing  by induction on $m=(\rho, \gamma^\vee-\mu^\vee)$, we shall prove that such an element of length 
$(\rho,\gamma^\vee-\mu^\vee)$ does exist.

\textbullet \quad if $m=1$, then it follows from Eq.~\eqref{eq:rho-gamma} that $\gamma$ and $\mu$ differ
only in a unique coordinate, say $c_{\tilde \ap}$, and then $\mu=\sigma_{\tilde\ap}(\gamma)$.

\textbullet \quad Assume that $m>1$ and $\gamma-\mu=\sum_{\ap\in \Pi}d_\ap\ap$ with 
$d_\ap\ge 0$. Then  $(\gamma-\mu,\ap)>0$ for some $\ap\in \Pi$ and hence $d_\ap>0$. Note that
$r$ divides $d_\ap$ whenever $\ap$ is short.

(a) \ If $(\gamma,\ap)>0$, then $(\rho, (\sigma_\ap(\gamma)^\vee))=(\rho,\gamma^\vee)-1$ and
$\gamma\succ \sigma_\ap(\gamma)\curge \mu$.  

(b) \ If $(\mu,\ap)<0$, then $(\rho, (\sigma_\ap(\mu)^\vee))=(\rho,\mu^\vee)+1$ and 
$\gamma\curge \sigma_\ap(\mu) \succ \mu$.
\\
This yields the induction step and existence of $w$ with $\ell(w)=(\rho,\gamma^\vee-\mu^\vee)$.

{\sf (3)} \ If $w(\gamma)=\mu$ and $\ell(w)=(\rho,\gamma^\vee-\mu^\vee)$, then
$ww_{\theta,\gamma}$ takes $\theta$ to $\mu$ and $\ell(ww_{\theta,\gamma})\le \ell(w)+
\ell(w_{\theta,\gamma})=(\rho,\theta^\vee-\mu^\vee)$.  Therefore, $ww_{\theta,\gamma}=w_{\theta,\mu}$ 
and $w=w_{\theta,\mu}w_{\theta,\gamma}^{-1}$ is the unique element of minimal length taking $\gamma$ to
$\mu$. Note also that $\ell(w_{\theta,\mu})=\ell(w_{\gamma,\mu})+\ell(w_{\theta,\gamma})$.
\end{proof}

\begin{rmk}   \label{rem:passage-dual}
1) \ Using the passage to the dual root system, one can get a version of Proposition~\ref{prop:unique-short} for 
$\gamma,\mu\in \Delta^+_s$.  Here $\ell(w_{\gamma,\mu})=(\rho^\vee,\theta-\mu)=\hot(\gamma)-\hot(\mu)$.
\\ \indent
2) \ For our construction of \MS, we only need this Proposition with $\mu\in\Pi_l$.
\\ \indent
3) \ The inversion set $\eus N(w_{\theta,\mu}^{\ -1})$ has an explicit description, see \eqref{eq:w-inverse}.
In general, we have $w_{\theta,\mu}^{\ -1}=w_{\theta,\gamma}^{\ -1}
w_{\gamma,\mu}^{\ -1}$ and 
$\ell(w_{\theta,\mu}^{\ -1})=\ell(w_{\theta,\gamma}^{\ -1})+\ell(w_{\gamma,\mu}^{\ -1})$. Therefore,
$\eus N(w_{\gamma,\mu}^{\ -1})\subset \eus N(w_{\theta,\mu}^{\ -1}) $.
However, there is no such a description for $\eus N(w_{\theta,\mu})$.
\end{rmk}

Let $w_{\gamma,\mu}=\sigma_{i_m}\!\cdots \sigma_{i_1}$ be a reduced expression. Here 
$(\ap_{i_{k}}, \sigma_{i_{k-1}}\!\cdots \sigma_{i_1}(\gamma))>0$ for
$k=1,2,\dots,m$ and $m=(\rho, \gamma^\vee-\mu^\vee)$. The root sequence 
\[
\gamma_0=\gamma, \ \gamma_1=\sigma_{i_1}(\gamma_0), \ \gamma_2=\sigma_{i_2}(\gamma_1), \ \dots, \ 
\gamma_m=\sigma_{i_m}(\gamma_{m-1})=\mu
\]  
is a {\it path of minimal length\/} connecting $\gamma$ and $\mu$. It follows from 
Proposition~\ref{prop:unique-short} that all paths of minimal length yield one and the same element of 
$W$, i.e., provide different reduced expressions for $w_{\gamma,\mu}$. 
If $\ap_{i_k}\in \Pi_l$, then $\gamma_k=\gamma_{k-1}-\ap_{i_k}$; and if 
$\ap_{i_k}\in \Pi_s$, then $\gamma_k=\gamma_{k-1}-r\ap_{i_k}$. (Recall that $\gamma$ is supposed to
be long.) Therefore, the multiplicities of simple reflections occurring in a reduced expression for 
$w_{\gamma,\mu}$ are fully
determined by $\gamma-\mu$. That is, if $\gamma-\mu=\sum_{\ap\in\Pi_l}a_\ap\ap+
\sum_{\beta\in\Pi_s}b_\beta\beta$, then each $\sigma_\ap$ ($\ap\in\Pi_l$) occurs $a_\ap$ times and
$\sigma_\beta$ ($\beta\in\Pi_s$) occurs $b_\beta/r$ times in a reduced expression for $w_{\gamma,\mu}$.

In the {\sf ADE} case, we have $r=1$ and $\gamma_k-\gamma_{k-1}$ is always a simple root. Therefore,
$\gamma_k$ covers $\gamma_{k-1}$ and
$(\gamma_0, \ \gamma_1, \dots, \gamma_m)$ is a true path in the Hasse diagram 
of $(\Delta^+, \curge)$. Since $r=1$, we also have
$(\rho,\gamma^\vee)=\hot(\gamma)$ and $\ell(w_{\gamma,\mu})=\hot(\gamma)-\hot(\mu)$. 

Following C.K.\,Fan~\cite{fan-these}, we give the following
\begin{df}  
An element $w\in W$ is said to be {\it commutative}, if no reduced expression for $w$ contains a 
substring $\sigma_\ap \sigma_\beta \sigma_\ap$, where $\ap$ and $\beta$ are adjacent simple roots.
\end{df}

Every reduced word for $w\in W$ can be obtained from any other
by applying a sequence of braid relations~\cite[Ch.\,4, \S 1.5]{bour}. Since we have noticed above that 
all reduced expressions for $w_{\gamma,\mu}$ have the same distribution of multiplicities of simple 
reflections, the elements of the form $w_{\gamma,\mu}$ are commutative in the simply-laced case. Below, we provide a more elementary proof of this assertion. 

\begin{lm}   \label{lm:commut-elem}
Suppose that $\Delta$ is simply-laced. Then
for all $\gamma,\mu\in\Delta^+$ such that $\gamma\curge\mu$, the element $w_{\gamma,\mu}$ is commutative.
\end{lm}
\begin{proof}
Assume not, and let $\sigma_\ap \sigma_\beta \sigma_\ap$ be a substring of a reduced expression for 
$w_{\gamma,\mu}$, with adjacent $\ap,\beta\in\Pi$. Then a path of minimal length connecting $\gamma$ and $\mu$  contains a sub-path
\[
   \dots, \ \nu, \ \sigma_\ap(\nu)=\nu-\ap, \ \sigma_\beta(\nu-\ap)=\nu-\ap-\beta, \ \nu-2\ap-\beta , \ \dots
\]
Here $(\ap,\beta)=-1$, $(\nu,\ap)=1$, and $(\nu-\ap,\beta)=1$. Then
$(\nu-\ap-\beta,\ap)=0$, which implies that $\nu-2\ap-\beta\not\in\Delta$. A contradiction!
\end{proof}
\begin{rmk}
Elements $w_{\gamma,\mu}$ are not always commutative outside the {\sf ADE}-realm. For instance, if $\Delta$
is of type $\GR{F}{4}$ and $\Pi=\{\ap_1,\dots,\ap_4\}$, with numbering as in \cite{t41}, then 
$w=\sigma_3\sigma_2\sigma_3\sigma_4$ is the shortest element taking $\theta=2\ap_1+4\ap_2+3\ap_3+2\ap_4$ to $\mu=2\ap_1+2\ap_2+\ap_3+\ap_4$.
\\ \indent
In general, we can prove that the elements of the form $w_{\gamma,\mu}$, where $\gamma\curge\mu$ and $\|\gamma\|=\|\mu\|$, are {\it fully commutative} in the sense of Stembridge \cite{stembr}. But this is not 
needed here.   
\end{rmk}
In the simply-laced case, 
the following characterisation of commutative elements is given in~\cite[Theorem\,2.4]{FS}:
\beq  \label{eq:svoistvo}
\text{\it $w\in W$ is commutative if and only if for all $\gamma,\gamma'\in \eus N(w)$, one has $(\gamma,\gamma')\ge 0$.}
\eeq

Given $w\in W$, we say that $w'$ is a {\it left factor\/} of $w$, if $\ell(w'^{-1}w)=\ell(w)-\ell(w')$.
In other words, $w=w'w''$ and $\ell(w)=\ell(w')+\ell(w'')$. Similarly, $w''$ is a {\it right factor\/} of $w$, if 
$\ell(ww''^{-1})=\ell(w)-\ell(w'')$. If  $w=w'w''$, then $w'$ is a {left factor\/} of $w$ if and only if 
$w''$ is a {right factor\/} of $w$. In that case, $\eus N(w)=\eus N(w'')\sqcup (w'')^{-1}\eus N(w')$. In
particular, $\eus N(w'')\subset \eus N(w)$.
 
\begin{prop}   \label{lm:lezhit-nizhe}
Suppose that $w$ is commutative and $w'$ is a {left factor} of $w$. Then 
$w'$ is commutative, too, and
$\eus N(w')$ is ``located below'' $\eus N(w)$ with respect to ``$\curge$''. The latter means that, for any $\gamma'\in \eus N(w')$, there exists
$\gamma\in\eus N(w)$ such that $\gamma'\curle\gamma$. 
\end{prop}
\begin{proof}
The commutativity of $w'$ is obvious. 

\textbullet \quad If $\ell(w'')=1$, then $w=w'\sigma_\ap$ for some $\ap\in\Pi$.
Here $\eus N(w)=\sigma_\ap(\eus N(w'))\cup\{\ap\}$.
Take any $\mu'\in \eus N(w')$. Then $\mu=\sigma_\ap(\mu')\in \eus N(w)$ and $(\ap,\mu)\ge 0$ in view of 
Eq.~\eqref{eq:svoistvo}. Therefore, $\mu'=\sigma_\ap(\mu)\curle \mu$.

\textbullet \quad In general, we argue by downward induction on $\ell(w'')$. If $\ell(w'')>1$, then 
$w''=u\sigma_\ap$, where $\ell(u)=\ell(w'')-1$. Now, $w=(w'u)\sigma_\ap$ and the previous argument shows
that  $\eus N(w'u)$ is located below $\eus N(w)$. By the induction assumption, $\eus N(w')$ is located below
$\eus N(w'u)$, and we are done.
\end{proof}

\begin{ex}  \label{ex:left-factor}
We are going to apply the above proposition in the following situation. Suppose that 
$\gamma\succ \mu\succ\nu$ are long roots. Then $w_{\gamma,\nu}=w_{\mu,\nu}w_{\gamma,\mu}$ and 
$w_{\mu,\nu}$ is a left factor of $w_{\gamma,\nu}$. In the simply-laced case, all these elements are commutative and we conclude that $\eus N(w_{\mu,\nu})$ is located below $\eus N(w_{\gamma,\nu})$.
\end{ex}

\section{\MS\ associated with a maximal abelian ideal}
\label{sect:dva}

\noindent 
In this section, we fix  $\ap\in\Pi_l$ and consider the corresponding maximal abelian ideal
$I(\ap)_{max}\in \Ab_\ap$. Eventually, we stick to the {\sf ADE} case (i.e., $\Pi=\Pi_l$), but 
Lemmas~\ref{lm:N(w_theta,ap)} and \ref{lm:summa}  are valid in the general setting.
Recall that $I(\ap)_{min}=I(\ap)_{max}\cap \gH$.  
In the special case $\mu=\ap$, there is the following complement  
to Lemma~\ref{lm:mu-min}.

\begin{lm}    \label{lm:N(w_theta,ap)}
For  $\ap\in\Pi_l$, one has  
$\eus N(w_{\theta,\ap})=\gH\setminus I(\ap)_{max}=\gH\setminus I(\ap)_{min}$. In other words, 
$\gH=I(\ap)_{min}\sqcup \eus N(w_{\theta,\ap})$. Furthermore, $\mu\in \eus N(w_{\theta,\ap})$ if and only if
$\theta-\mu\in I(\ap)_{min}\setminus \{\theta\}$.
\end{lm}
\begin{proof}
By \cite[Sect.\,4]{imrn}, we have $\# I(\ap)_{min}=\ell(w_{\theta,\ap})+1=(\rho,\theta^\vee)$
(see also Lemma~\ref{lm:mu-min}(2)). On the other
hand, 
\[
    (2\rho,\theta^\vee)=\sum_{\gamma\in\Delta^+}(\gamma,\theta^\vee)=
    \#\{\gamma\mid (\gamma,\theta^\vee)=1\}+(\theta,\theta^\vee)=(\#\gH-1)+2=\#\gH+1 .
\]
Therefore $\# I(\ap)_{min}+ \#\eus N(w_{\theta,\ap})= 2(\rho,\theta^\vee)-1=\#\gH$. Since both 
$I(\ap)_{min}$ and $\eus N(w_{\theta,\ap})$ lies in $\gH$, comparing with Lemma~\ref{lm:mu-min} shows that the only possibility is that $\gH$ is a disjoint union of $I(\ap)_{min}$ and $\eus N(w_{\theta,\ap})$.
Indeed, if $\nu\in I(\ap)_{min}\cap\eus N(w_{\theta,\ap})$, then  $\theta-\nu\in I(\ap)_{min}$, which 
contradicts the fact that $I(\ap)_{min}$ is abelian.
\\ \indent
The last assertion stems from the fact that $\gH\setminus \{\theta\}$ is the union of pairs of the form
$\{\mu, \theta-\mu\}$.
\end{proof}

\begin{lm}  \label{lm:summa}
For any $\mu'\in \Delta^+\setminus I(\ap)_{max}$, there exists $\mu\in I(\ap)_{min}$ such that $\mu+\mu'$
is a root; and thereby $\mu+\mu'\in I(\ap)_{min}$.
\end{lm}
\begin{proof}
If $\mu'\in \max(\Delta^+\setminus I(\ap)_{max})$, then $\mu'\in\gH$ and $\mu=\theta-\mu'\in
I(\ap)_{min}$~\cite[Theorem\,4.7]{jems}. That is, the required property holds for the maximal elements of
$\Delta^+\setminus I(\ap)_{max}$. 

Assume that the required property is not satisfied for all of $\Delta^+\setminus I(\ap)_{max}$.
Let $\nu\in \Delta^+\setminus I(\ap)_{max}$ be a maximal root among those that do not have the required property. Since  $\nu\not\in \max(\Delta^+\setminus I(\ap)_{max})$, there is $\beta\in\Pi$ such that
$\nu+\beta\in \Delta^+\setminus I(\ap)_{max}$.   Then $\nu+\beta$ has the required property and one
can find $\mu\in I(\ap)_{min}$ such that
$(\nu+\beta)+\mu$ is a root. Here one easily proves that $\nu+\mu$ or $\beta+\mu$ is a root. In the 
second case, one obtains $\beta+\mu\in I(\ap)_{min}$. Therefore, both possibilities contradict  the 
assumption on $\nu$.   Thus, all roots in  $\Delta^+\setminus I(\ap)_{max}$ must have the required property. 
\end{proof}

For any $\gamma\in I(\ap)_{max}$, one has $\gamma\curge\ap$ \cite[Theorem\,3.5]{jems}.
Therefore, using Proposition~\ref{prop:unique-short}, we obtain the shortest element 
$w_{\gamma,\ap}$ for each {\sl long\/} $\gamma$. 
\\ \indent
In the rest of the section, we assume that $\Delta$ is simply-laced.

\begin{lm}  \label{lm:lezhit-vne-ideala}
For any $\gamma\in I(\ap)_{max}$, we have $\eus N(w_{\gamma,\ap})\subset  
\Delta^+\setminus I(\ap)_{max}$.
\end{lm}
\begin{proof}
If $\gamma=\theta$, then $\eus N(w_{\theta,\ap})=\gH\setminus I(\ap)_{max}$ by Lemma~\ref{lm:N(w_theta,ap)}.
For arbitrary $\gamma\in I(\ap)_{max}$, we have $\theta\curge\gamma\curge\ap$ and 
$w_{\theta,\ap}=w_{\gamma,\ap}w_{\theta,\gamma}$ (see the proof of Proposition~\ref{prop:unique-short}).
Then Proposition~\ref{lm:lezhit-nizhe} and Example~\ref{ex:left-factor} show that $\eus N(w_{\gamma,\ap})$ is located
below  $\eus N(w_{\theta,\ap})$. As $I(\ap)_{max}$ is an ideal and $\eus N(w_{\theta,\ap})$ does not intersect  $I(\ap)_{max}$, so
does $\eus N(w_{\gamma,\ap})$.
\end{proof}

Now, we are ready to provide an accurate construction of \MS\ associated with $\ap\in\Pi$.
Having the shortest elements $w_{\gamma,\ap}$ ($\gamma\in I(\ap)_{max}$) at our disposal, we set
$\tilde w_\gamma=\sigma_\ap w_{\gamma,\ap}$.  Then
$\eus N(\tilde w_\gamma)=\eus N(w_{\gamma,\ap})\cup\{w_{\gamma,\ap}^{-1}(\ap)\}=
\eus N(w_{\gamma,\ap})\cup\{\gamma\}$. In particular, 
$\ell(\tilde w_\gamma)=\ell(w_{\gamma,\ap})+1=\hot(\gamma)$. Note also that $\tilde w_\gamma$ is the shortest element of $W$ taking $\gamma$ to $-\ap$.

\begin{thm}   \label{thm:main}
Suppose that  $\Delta$ is simply-laced. For $\ap\in\Pi$ and the corresponding
maximal abelian ideal $I(\ap)_{max}$, the family
$\eus F_\ap=\{\tilde w_\gamma\mid \gamma\in I(\ap)_{max}\}$ is a \MS\ in $W$ and\/ 
$\ess(\eus F_\ap)\supset I(\ap)_{max}$.
\end{thm}
\begin{proof}
We already know that  
$\eus N(\tilde w_\gamma)=\eus N(w_{\gamma,\ap})\cup\{\gamma\}$ and
\vskip.7ex
\centerline{$\eus N(w_{\gamma,\ap})\subset \Delta^+\setminus I(\ap)_{max}$ \ (Lemma~\ref{lm:lezhit-vne-ideala}).}
\vskip.7ex
Consequently, 
\[
\bigcup_{\gamma\in I(\ap)_{max}}\eus N(\tilde w_\gamma)=
I(\ap)_{max}\sqcup \Bigl(\bigcup_{\gamma\in I(\ap)_{max}} \eus N(w_{\gamma,\ap})\Bigr) .
\]
Take an arbitrary $\mu'\in \Delta^+\setminus I(\ap)_{max}$. By Lemma~\ref{lm:summa}, there is $\mu\in
I(\ap)_{min}$ such that $\gamma:=\mu+\mu'$ is a root, and thereby 
$\gamma\in I(\ap)_{min}\subset I(\ap)_{max}$. Then
\[
   \ap=w_{\gamma,\ap}(\gamma)=w_{\gamma,\ap}(\mu)+w_{\gamma,\ap}(\mu') .
\]
By Lemma~\ref{lm:lezhit-vne-ideala},  $w_{\gamma,\ap}(\mu)$ is a positive root. Hence 
$w_{\gamma,\ap}(\mu')$ must be negative, i.e., $\mu'\in \eus N(w_{\gamma,\ap})$. It follows that 
$\displaystyle \bigcup_{\gamma\in I(\ap)_{min}} \eus N(w_{\gamma,\ap})=\Delta^+\setminus I(\ap)_{max}$ and hence
$\displaystyle \bigcup_{\gamma\in I(\ap)_{max}} \eus N(\tilde w_\gamma)=\Delta^+$.
Note that $\tilde w_\gamma$ is the only element of $\eus F_\ap$ whose inversion set contains 
$\gamma\in I(\ap)_{max}$. Therefore, $\eus F_\ap$ is a \MS\ and $I(\ap)_{max}\subset \ess(\eus F_\ap)$.
\end{proof}

\begin{rema}  Our use of the canonical elements $w_{\gamma,\ap}$ and $\tilde w_\gamma$ ($\gamma\in I(\ap)_{max}$) replaces the ad hoc considerations and case-by-case calculations of \cite{MMOP}.
\end{rema}
The inclusion $\ess(\eus F_\ap)\supset I(\ap)_{max}$ can be strict.

\begin{ex}   \label{ex:An}
Let $\Delta$ be of type $\GR{A}{n}$ and $\Delta^+=\{\esi_i-\esi_j\mid 1\le i< j\le n+1\}$. Here $\theta=\esi_1-\esi_{n+1}$ and $\Pi=\{\ap_1,\dots,\ap_n\}$ with $\ap_i=\esi_i-\esi_{i+1}$. 

Consider the \MS\ associated with $\ap_1$. In this case, 
\[
I(\ap_1)_{max}=I(\ap_1)_{min}=\{\ap_1,\ap_1+\ap_2,\dots,\ap_1+\dots+\ap_n=\theta\} .
\] 
Then $w_{\theta,\ap_1}=\sigma_2\cdots \sigma_n$,
$\tilde w_\theta=\sigma_1\sigma_2\cdots \sigma_n$, and 
$\tilde w_{\ap_1+\dots+\ap_i}=\sigma_1\cdots \sigma_i$ for $1\le i\le n$. Thus,
\[
\eus N(\tilde w_{\ap_1+\dots+\ap_i})=\{\ap_i, \ap_i+\ap_{i-1},\dots, \ap_i+\dots+\ap_1\}
\]
and $\Delta^+=\bigsqcup_{i=1}^n \eus N(\tilde w_{\ap_1+\dots+\ap_i})$ (the disjoint union!).
Therefore, $\ess(\eus F_{\ap_1})=\Delta^+$. 

The similar situation occurs for $\ap_n$.
On the other hand, if $2\le i \le n-1$, then $(\theta,\ap_i)=0$ and therefore
$\ess(\eus F_{\ap_i})=I(\ap_i)_{max}$, see Theorem~\ref{thm:ab-nilrad}(ii) below.
\end{ex}

\begin{rmk}
If $\Delta$ has two root lengths and $\mu\in I(\ap)_{max}$ is short, then $\ap$ and $\mu$ are not $W$-conjugate. Therefore, one cannot define an element  $w_{\mu,\ap}$ and our approach, as it is, fails.
Of course, for each $\mu\in I(\ap)_{max}\cap \Delta_s$, one can find $\beta\in\Pi_s$ such that $\mu\curge\beta$ 
and then take the shortest element $w_{\mu,\beta}$, etc. But this is not canonical and, moreover, this 
merely does not provide a \MS!
In general, Malvenuto et al. suggest that the maximum cardinality of a \MS\ 
equals the maximum cardinality of a strongly abelian subset of $\Delta^+$~\cite[Sect.\,9]{MMOP}. 
Here $\ca\subset\Delta^+$
is said to be {\it strongly abelian}, if $(\BR_{>0}\gamma+\BR_{>0}\mu) \cap\Delta=\varnothing$ for any 
$\gamma,\mu\in\ca$.
\end{rmk}
\begin{ex}    \label{ex:F4}
For $\Delta$ of type $\GR{F}{4}$, $\#\ca\le 6$ and there is a \MS\ of cardinality $6$ \cite[(7.1)]{MMOP}, 
whereas the maximal abelian ideals are of cardinality $8$ and $9$. From our point of view, that specific 
\MS\ can be obtained as follows. Write $[n_1n_2n_3n_4]$ for $\gamma=\sum_{i=1}^4 n_i\ap_i\in \Delta^+(\GR{F}{4})$. The numbering of $\Pi$ is as in \cite[Tables]{t41}, so that $\Pi_s=\{\ap_1,\ap_2\}$ and
$\Pi_l=\{\ap_3,\ap_4\}$. Then
\[
 \ca=\{[2432], [2431], [2421], [2321], [1321], [1221] \}
\]
is a maximal strongly abelian subset, where the last three roots are short. For $\gamma\in \ca_l$
(resp. $\gamma\in \ca_s$), we take the element of minimal length $w_{\gamma,\ap_3}$ (resp. 
$w_{\gamma,\ap_1}$) and set $\tilde w_\gamma=\sigma_3w_{\gamma,\ap_3}$ ($\gamma\in \ca_l$) and 
$\tilde w_\gamma=\sigma_1w_{\gamma,\ap_1}$ ($\gamma\in \ca_s$).
The resulting six elements form a MICS pointed out in \cite{MMOP}. This procedure strongly resembles our
construction of \MS\ in Theorem~\ref{thm:main}, but we do not have a general theoretical explanation for this. 
Since this construction exploits simple
roots $\ap_1$ and $\ap_3$, we denote the resulting \MS\ by $\eus F_{1,3}$. Furthermore, a similar construction 
can be performed with the other simple roots in suitable combinations. That is,  we use either
$\ap_3$ or $\ap_4$ for $\gamma\in\ca_l$ and either $\ap_1$ or $\ap_2$ for $\gamma\in\ca_s$. This yields four different \MS\ of cardinality six:
$\eus F_{1,3}, \eus F_{1,4}, \eus F_{2,3}, \eus F_{2,4}$. But the essential sets for them are essentially different! We have $\#\ess(\eus F_{1,3})=\#\ess(\eus F_{2,4})=12$, $\#\ess(\eus F_{2,3})=14$, and
$\#\ess(\eus F_{1,4})=10$. For instance, $\ess(\eus F_{1,4})=\ca\cup\{\ap_1,\ap_4, [1211],[2221]\}$.
Thus, $\eus F_{1,4}$ has the smallest defect among them (=4).
\end{ex}
 
\section{Some properties of root systems}
\label{sect:some-prop-roots}

\noindent Here we gather  preparatory results on long roots that are needed in the following 
section, where we  describe the essential set for some \MS\ constructed in Section~\ref{sect:dva}.
\subsection{} Given $\gamma\in\Delta^+$, consider the set
$\Gamma_\gamma=\{\nu\in\Delta^+ \mid \gamma-\nu\in \Delta^+\}$. Clearly, if $\nu\in\Gamma_\gamma$, then $\gamma-\nu\in \Gamma_\gamma$ as well and $\gamma-\nu\ne \nu$. Therefore,
$\#\Gamma_\gamma  $ is even, and if it is $2k$, then there are exactly $k$ different ways to write $\gamma$ as a
sum of two positive roots.

\begin{lm}   \label{lem:svyaz-s-otrazheniyami}
If $\gamma\in\Delta^+_l$ and $\sigma_\gamma\in W$ is the reflection with respect to $\gamma$, then
$\eus N(\sigma_\gamma)=\Gamma_\gamma\cup\{\gamma\}$. In particular, \ 
$\#\Gamma_\gamma  +1=\ell(\sigma_\gamma)$.
\end{lm}
\begin{proof} Since $\sigma_\gamma(\nu)=\nu-(\nu,\gamma^\vee)\gamma$ \ and \ $(\nu,\gamma^\vee)=2$
if and only if $\nu=\gamma$, we have
\begin{multline*}
\eus N(\sigma_\gamma)=
\{\gamma\} \cup \{\mu\in \Delta^+\mid (\mu,\gamma^\vee)=1 \ \& \ \mu-\gamma\in -\Delta^+ \}=
\\
\{\gamma\} \cup \{\mu\in \Delta^+\mid  \gamma-\mu \in \Delta^+ \}= \{\gamma\} \cup\Gamma_\gamma.
\qedhere
\end{multline*}  
\end{proof}

\begin{rema}
If $\Delta$ is of type {\sf BCF} and $\gamma$ is short, then there is still a natural bijection between 
$\eus N(\sigma_\gamma)$ and $\Gamma_\gamma\cup\{\gamma\}$, hence the equality $\#\Gamma_\gamma  +1=\ell(\sigma_\gamma)$ survives. But this is not true for the dominant short root in type $\GR{G}{2}$.
\end{rema}

\begin{df}[\cite{bh93}]
The {\it depth\/} of  $\gamma\in\Delta^+$ is \ $\mathrm{dp}(\gamma)=\min\{\ell(w)\mid w\in W, \ w(\gamma)\in -\Delta^+\}$.
\end{df}

\noindent 
This notion has been studied in the context of arbitrary (possibly infinite) Coxeter groups~\cite{bh93,deMan},
but we only use it for Weyl groups. It is easily seen that if the required minimum of $\ell(w)$ is achieved, then 
$w(\gamma)\in-\Pi$. Furthermore, if $w(\gamma)=-\ap\in -\Pi$, then $w=\sigma_\ap w'$, where
$w'(\gamma)=\ap$ and $\ell(w)=\ell(w')+1$. Combining this with Proposition~\ref{prop:unique-short} with 
$\mu=\ap$, we see that $w'=w_{\gamma,\ap}$. Hence

\begin{lm}
If $\Delta$ is simply-laced, then $\mathrm{dp}(\gamma)=(\rho,\gamma^\vee)=\hot(\gamma)$ for
all \ $\gamma\in \Delta^+$.
\end{lm}

Then using the general equality $\ell(\sigma_\gamma)=2\mathrm{dp}(\gamma)-1$~\cite[Lemma\,1.2]{deMan},
we obtain  
\[
  \#\Gamma_\gamma=\ell(\sigma_\gamma)-1=2\mathrm{dp}(\gamma)-2=2\hot(\gamma)-2 .
\]
In other words,
\begin{prop}  \label{prop:rozlozh-gamma}
If $\Delta$ is simply-laced, then there are \ $\hot(\gamma)-1$ ways to write $\gamma\in\Delta^+$ as a sum of two positive roots.
\end{prop}
\begin{rema}
More generally, if $\Delta$ is arbitrary and $\gamma\in\Delta^+_s$, then
$\mathrm{dp}(\gamma)=\hot(\gamma)$. This allows to conclude that if $\gamma\in\Delta^+$ is also arbitrary, then
$\mathrm{dp}(\gamma)=\min\{\hot(\gamma),\hot(\gamma^\vee)\}$, where $\hot(\gamma^\vee)$ is the height
of $\gamma^\vee$ in the dual root system $\Delta^\vee$. But we do need this here.
\end{rema}
\subsection{} Consider a sort of dual object to $\Gamma_\gamma$.
For $\gamma\in\Delta^+$,  set
$\Phi_\gamma=\{\nu \in\Delta^+ \mid \nu-\gamma\in \Delta^+\}$.

\begin{lm}   \label{lm:Phi-gamma}
If $\gamma\in \Delta^+_l$, then $\#\Phi_\gamma=(\rho,\theta^\vee-\gamma^\vee)$. In particular, if $\Delta$
is simply-laced, then $\#\Phi_\gamma=\hot(\theta)-\hot(\gamma)=h-1-\hot(\gamma)$.
\end{lm}
\begin{proof}
Consider the shortest element $w_{\theta,\gamma}\in W$. Then 
$\ell(w_{\theta,\gamma})=(\rho,\theta^\vee-\gamma^\vee)$ (Proposition~\ref{prop:unique-short})
and $\eus N(w_{\theta,\gamma}^{\ -1})=\{\mu\in\Delta^+\mid (\mu,\gamma^\vee)=-1\}$, 
see~Eq.~\eqref{eq:w-inverse}. That is, the number of positive roots
$\mu$ such that $\gamma+\mu$ is a root equals $\ell(w_{\theta,\gamma})$. It remains to observe that
$(\mu,\gamma^\vee)=-1$ if and only if $\gamma+\mu\in \Phi_\gamma$.  
\end{proof} 

\subsection{} For $\ap\in \Pi$ and $\nu\in\Delta$, let $[\nu:\ap]$ denote the coefficient of $\ap$ in the sum
$\nu=\sum_{\beta\in \Pi}c_\beta\beta$. We also say that $[\nu:\ap]$ is the $\ap$-{\it height\/} of $\nu$.
Set $\Delta_\ap(i)=\{\gamma\in\Delta \mid [\gamma:\ap]=i\}$ and 
$\Delta_\ap(0)^+=\Delta_\ap(0)\cap \Delta^+$. The Weyl group of $\Delta_\ap(0)$ is 
$W_\ap(0)=\langle \sigma_\beta\mid \beta\in \Pi\setminus\{\ap\}\rangle$.
\\ \indent
Although, we are not aware of a general description for $\eus N(w_{\gamma,\ap})$, we do have a nice
formula if $[\gamma:\ap]=1$, i.e., $\gamma\in\Delta_\ap(1)$.

\begin{prop}   \label{prop:inv-set-for-height-1}
Suppose that $\Delta$ is simply-laced.
If $\gamma\curge \ap$ and $[\gamma:\ap]=1$, then 
$\eus N(w_{\gamma,\ap})=\{\mu\in \Delta_\ap(0)^+\mid \gamma-\mu\in\Delta^+\}$.
\end{prop}
\begin{proof}
Since $\gamma$ and $\ap$ have the same $\ap$-height, any reduced expression for $w_{\gamma,\ap}$
does not contain $\sigma_\ap$. Therefore, $w_{\gamma,\ap}\in W_\ap(0)$ and 
$\eus N(w_{\gamma,\ap})\subset \Delta_\ap(0)^+$. Suppose that $\gamma=\mu_1+\mu_2$ is a sum of positive 
roots. Then  $\mu_1\in \Delta_\ap(0)^+$ and $\mu_2\in\Delta_\ap(1)$. Since
\[
  \ap=w_{\gamma,\ap}(\gamma)=w_{\gamma,\ap}(\mu_1)+w_{\gamma,\ap}(\mu_2) ,
\]
exactly one summand in the right hand side is negative. The above discussion implies that
$w_{\gamma,\ap}(\mu_1)\in -\Delta^+$, i.e.,  $\mu_1\in \eus N(w_{\gamma,\ap})$. This means that every
presentation of $\gamma$ as a sum of two positive roots yields an element of $\eus N(w_{\gamma,\ap})$.
On the other hand, we know that 
\\ \indent
(i) \  $\#\eus N(w_{\gamma,\ap})=\hot(\gamma)-1$ (Proposition~\ref{prop:unique-short})  and
\\ \indent
(ii) \ the total number of such presentations is $\hot(\gamma)-1$ (Proposition~\ref{prop:rozlozh-gamma}). 
\\
Thus, each element of $\eus N(w_{\gamma,\ap})$ lies in $\Delta_\ap(0)^+$ and
is associated with a presentation of $\gamma$ as a sum of two roots. 
\end{proof}

\section{On the essential set of $\eus F_\ap$}
\label{sect:essential}

\noindent
Any \MS\ $\eus F\subset W$ determines a natural statistic on $\Delta^+$:
\[
    \gamma\mapsto 
    n_{\eus F}(\gamma)=\#\{w\in \eus F\mid \gamma\in \eus N(w)\} .
\]
We call it the $\eus F$-{\it statistic\/} on $\Delta^+$ and say that $n_{\eus F}(\gamma)$ is the 
$\eus F$-{\it multiplicity\/} of $\gamma$. Clearly, $\ess(\eus F)=\{\gamma\in\Delta^+\mid n_{\eus F}(\gamma)=1\}$. 
For $\eus F=\eus F_\ap$ as above, we write $n_{\ap}(\gamma)$ in place of 
$n_{\eus F_\ap}(\gamma)$.

In this section,  
we describe $\ess(\eus F_\ap)$ in the following two cases:

{\sf (A)} \  $I(\ap)_{max}$ is the set of roots of the nilradical of a standard parabolic subalgebra $\p$ of $\g$.
(That is, the nilradical of $\p$ is abelian.)

{\sf (B)} \ $\theta$ is a fundamental weight and $\ap=\hat\ap$ is the unique simple root such that
$(\theta,\hat\ap)\ne 0$. 
\\[.7ex]
We also provide a full description of the $\eus F_\ap$-statistic in case {\sf (A)}. Recall the necessary setup. 

\noindent
Let $\p_\ap$ be the maximal parabolic subalgebra associated with $\ap\in\Pi$. Then $\Delta_\ap(0)$
is the root system of the standard Levi subalgebra of $\p_\ap$ and $\bigcup_{i\ge 1}\Delta_{\ap}(i)$
is the set of roots of the  nilradical $\n_\ap$ of $\p_\ap$. The system $\Delta_\ap(0)$ can be 
reducible, and we need its partition into irreducible subsystems
$\Delta_\ap(0)=\bigsqcup_{j\in J} \Delta_\ap(0)_j$.

\un{For {\sf (A)}}: \ Suppose that $\ap$ has the property that $[\theta:\ap]=1$. Then
$\Delta^+=\Delta_\ap(0)^+\cup \Delta_\ap(1)$ and $\n_\ap$ is an abelian ideal of $\be$.
Moreover, $\Delta_\ap(1)=I(\ap)_{max}$. Note that 
$\ap\in \Delta_\ap(1)$.

As is well-known, the simple roots $\ap$ such that $[\theta:\ap]=1$ are:

\textbullet\quad $\ap_1,\dots,\ap_n$ for $\GR{A}{n}$;

\textbullet\quad $\ap_1,\ap_{n-1},\ap_n$ for $\GR{D}{n}$, $n\ge 4$;

\textbullet\quad $\ap_1,\ap_5$ for $\GR{E}{6}$ and $\ap_1$ for $\GR{E}{7}$; there are no such simple roots in $\GR{E}{8}$. (We use the numbering of simple roots from \cite[Tables]{t41}.)
\\
Furthermore, $\#J=2$ for $(\GR{A}{n}, \ap_i)$ with $2\le i\le n-1$ and $\#J=1$  in the other cases.

\begin{thm}   \label{thm:ab-nilrad}
Suppose that $[\theta:\ap]=1$ and 
$\Delta^+=\Delta_\ap(0)^+\sqcup \Delta_\ap(1)=\Delta_\ap(0)^+\sqcup I(\ap)_{max}$, as
above. Consider the \MS\ \ $\eus F_\ap=\{\tilde w_\gamma\mid \gamma\in \Delta_\ap(1)\}$ constructed in Section~\ref{sect:dva}.

{\sf (i)} \ For $\mu\in \Delta_\ap(0)^+$, the $\eus F_\ap$-multiplicity of $\mu$ depends only on the irreducible 
subsystem $\Delta_\ap(0)_j$ to which $\mu$ belongs. Namely, if $h$ (resp. $h_j$) is the Coxeter number of 
$\Delta$ (resp. $\Delta_\ap(0)_j$), then $n_{\ap}(\mu)=h-h_j$.

{\sf (ii)} \ In particular, if $(\theta,\ap)=0$ (which covers all\/ $\GR{D}{n}$ and\/ $\GR{E}{n}$ cases), 
then\/ $\ess(\eus F_\ap)=I(\ap)_{max}$ and thereby\/ $\mathrm{defect}(\eus F_\ap)=0$.
\end{thm}
\begin{proof}
(i) \ By Theorem~\ref{thm:main}, $n_\ap(\mu)=1$ for $\mu\in\Delta_\ap(1)$, and we only have to handle the 
case in which $\mu\in\Delta_\ap(0)^+$. The very construction of $\eus F_\ap$ shows that $n_\ap(\mu)$ 
equals the number of $\gamma\in\Delta_\ap(1)$ such that $\mu\in \eus N(w_{\gamma,\ap})$. According to 
Proposition~\ref{prop:inv-set-for-height-1},  the latter equals the number of $\gamma\in\Delta_\ap(1)$ such
that $\gamma-\mu\in\Delta^+$.

Recall that $\Phi_\mu=\{\gamma\in \Delta^+\mid \gamma-\mu\in\Delta^+\}$ and $\#\Phi_\mu=h-1-\hot(\mu)$
(Lemma~\ref{lm:Phi-gamma}). Here 
$\Phi_\mu=\Phi_\mu(0)\cup\Phi_\mu(1)$, according to whether $\gamma$ lies in $\Delta_\ap(0)^+$ or 
$\Delta_\ap(1)$, and $n_\ap(\mu)=\#\Phi_\mu(1)$. If $\gamma\in\Delta_\ap(0)^+$, then $\gamma$ and $\mu$ must belong to one and the 
same irreducible subsystem of $\Delta_\ap(0)$. Therefore, $\#\Phi_\mu(0)=h_j-1-\hot(\mu)$ for 
$\mu\in \Delta_\ap(0)_j$ and hence $\#\Phi_\mu(1)=h-h_j$.

(ii) \ If $(\ap,\theta)=0$ and $[\theta:\ap]=1$, then the highest root $\theta_j$ in $\Delta_\ap(0)_j$ has the height at
least two less than the height of $\theta$, i.e., $h-h_j\ge 2$. (Indeed, at the very first step down from $\theta$ to
$\theta_j$,
one subtracts certain $\beta\in \Pi$ such that $(\beta,\theta)\ne 0$, and one also must subtract $\ap$ once.)
\end{proof}

\begin{ex}   \label{ex:posle-teor}  For all $\ap\in\Pi$ with $[\theta:\ap]=1$, we point out the differences
$h-h_j$ and thereby 
$\eus F_\ap$-multiplicities of the roots in $\Delta_\ap(0)^+$. Recall that $h=h(\Delta)$.
\begin{enumerate}
\item For $\GR{E}{6}$ and $\ap=\ap_1$ or $\ap_5$, $\Delta_\ap(0)$ is of type $\GR{D}{5}$ and $h-h(\GR{D}{5})=4$;
\item For $\GR{E}{7}$ and $\ap=\ap_1$, $\Delta_\ap(0)$ is of type $\GR{E}{6}$ and $h-h(\GR{E}{6})=6$;
\item For $\GR{D}{n}$ and $\ap=\ap_1$, $\Delta_\ap(0)$ is of type $\GR{D}{n-1}$ and $h-h(\GR{D}{n-1})=2$;
\item For $\GR{D}{n}$ and $\ap=\ap_{n-1}$ or $\ap_n$, $\Delta_\ap(0)$ is of type $\GR{A}{n-1}$ and 
$h-h(\GR{A}{n-1})=n-2$;
\item For $\GR{A}{n}$ and $\ap=\ap_{i}$ ($2\le i\le n-1$), $\Delta_\ap(0)$ is of type $\GR{A}{n-i}\times\GR{A}{i-1}$.
Therefore, the $\eus F_{\ap_i}$-multiplicity of $\mu\in \Delta_\ap(0)^+$ is either $h-h(\GR{A}{n-i})=i$ or 
$h-h(\GR{A}{i-1})=n-i+1$.
\item For $\GR{A}{n}$ and $\ap=\ap_1$ or $\ap_n$, $\Delta_\ap(0)$ is of type $\GR{A}{n-1}$  and 
$h-h(\GR{A}{n-1})=1$. Therefore here $\ess(\eus F_{\ap_1})=\ess(\eus F_{\ap_n})=\Delta^+$, as already 
computed in Example~\ref{ex:An}.
\end{enumerate}
The last item represents the only possibilities in which $(\theta,\ap)\ne 0$.
\end{ex}

\un{For {\sf (B)}}: \ Suppose now that $\theta$ is a fundamental weight. Write $\hat\ap$ for the unique simple 
root that is not orthogonal to $\theta$.  By assumption, $(\theta,\hat\ap^\vee)=1$. Therefore $\hat\ap$ is long,
$(\hat\ap,\theta^\vee)=1$, and $(\beta,\theta^\vee)=0$ for $\beta\in\Pi\setminus\{\hat\ap\}$. 
Hence $2=(\theta,\theta^\vee)=[\theta:\hat\ap]$ and
$(\gamma,\theta^\vee)=[\gamma:\hat\ap]$ for any $\gamma\in\Delta^+$.
In this situation, $\gH=\Delta_{\hat\ap}(1)\cup\Delta_{\hat\ap}(2)$ and $\Delta_{\hat\ap}(2)=\{\theta\}$.
Clearly, $\theta-\hat\ap$ is the unique root covered by $\theta$ and $\theta\succ\theta-\hat\ap\succ \hat\ap$.
Since $\theta-\hat\ap\in\Delta^+_l$, the shortest element $\hat w:=w_{\theta-\hat\ap,\hat\ap}$ is well-defined.
By the very definition of $\hat w$, we have $\hat w(\theta-\hat\ap)=\hat\ap$. We are going to prove that 
$\hat w$ is an involution, and the next assertion yields a first step towards that goal.

\begin{lm}   \label{lm:first-for-inv}
$\hat w(\hat\ap)=\theta-\hat\ap$. 
\end{lm}
\begin{proof}
Since $\theta-\hat\ap$ and $\hat\ap$ have the same $\hat\ap$-height, a reduced expression for $\hat w$
does not contain $\sigma_{\hat\ap}$.  Therefore, $\hat w\in W_{\hat\ap}(0)=\langle \sigma_\beta\mid
\beta\in\Pi\setminus \{\hat\ap\}\rangle$ and $\hat w(\theta)=\theta$. Applying $\hat w$ to the equality
$(\theta-\hat\ap)+\hat\ap=\theta$, we obtain $\hat\ap+\hat w(\hat\ap)=\theta$, as required.
\end{proof}

\begin{prop}  \label{prop:hat-w-invol}
If $\theta$ is a fundamental weight, then $\hat w^2=1$.
\end{prop}
\begin{proof}
It suffices to prove that $\eus N(\hat w^{-1})\subset \eus N(\hat w)$. Since 
$\sigma_{\hat\ap}(\theta)=\theta-\hat\ap$,   $\hat w\sigma_{\hat\ap}=w_{\theta,\hat\ap}$ is the 
shortest element of $W$ taking $\theta$ to $\hat\ap$. Therefore, using Lemma~\ref{lm:first-for-inv}, we obtain
\[
\eus N(\hat w^{-1})=\eus N(w_{\theta,\hat\ap}^{\ -1})\setminus \{\hat w(\hat\ap)\}=
\eus N(w_{\theta,\hat\ap}^{\ -1})\setminus \{\theta-\hat\ap\} .
\]
Recall that $\eus N(w_{\theta,\hat\ap}^{\ -1})=\{\nu\in\Delta^+\mid \nu+\hat\ap\in \Delta^+\}$, see 
Eq.~\eqref{eq:w-inverse}. 
It follows that
\[
    \eus N(\hat w^{-1})=\{\nu\in\Delta^+\mid \nu+\hat\ap\in \Delta^+\setminus\{\theta\}\} .
\]
Because $\theta$ is the only root with $\hat\ap$-height equal to $2$, this means that $[\nu:\hat\ap]=0$.
That is, 
\[
    \eus N(\hat w^{-1})=\{\nu\in\Delta_{\hat\ap}(0)^+\mid \nu+\hat\ap\in \Delta_{\hat\ap}(1)\} .
\]
On the other hand, Proposition~\ref{prop:inv-set-for-height-1} with $\ap=\hat\ap$ shows that
\[
    \eus N(\hat w)=\{\mu\in\Delta_{\hat\ap}(0)^+\mid (\theta-\hat\ap)-\mu\in \Delta_{\hat\ap}(1)\} .
\]
Finally, if $\nu\in \eus N(\hat w^{-1})$, then $\nu+\hat\ap\in\Delta_{\hat\ap}(1)$, hence 
$\theta-(\nu+\hat\ap)$ is a root (in $\Delta_{\hat\ap}(1)$), i.e.,
$(\theta-\hat\ap)-\nu\in \Delta_{\hat\ap}(1)$. Thus, $\nu\in \eus N(\hat w)$ and we are done.
\end{proof}

\begin{thm}  \label{teor:ess-theta-fundam}
Suppose that $\Delta$ is simply-laced, $\theta$ is fundamental, and $\hat\ap$ is the unique simple root that is not orthogonal to 
$\theta$. Then\/ $\ess (\eus F_{\hat\ap})=\gH$ and $\mathrm{defect}(\eus F_{\hat\ap})=h-2$. 
\end{thm}
\begin{proof}
Recall that $\gH=I(\hat\ap)_{min}\sqcup \eus N(w_{\theta,\hat\ap})$ (Lemma~\ref{lm:N(w_theta,ap)}).
Since $\hat\ap\in\gH$, we have $I(\hat\ap)_{max}=I(\hat\ap)_{min}$ \cite[Theorem\,5.1]{imrn} and write 
below $\ihp$ for this common ideal.
\\ \indent 
{\bf 1)} \ We first prove that $\gH \subset \ess (\eus F_{\hat\ap})$. 
As $I(\hat\ap)\subset \ess (\eus F_{\hat\ap})$ (Theorem~\ref{thm:main}), our  task is to prove 
that $\eus N(w_{\theta,\hat\ap})\subset \ess (\eus F_{\hat\ap})$.

The proof of Theorem~\ref{thm:main} shows that $\bigcup_{\gamma\in\ihp} \eus N(w_{\gamma,\hat\ap})=
\Delta^+\setminus \ihp$ and the $\eus F_{\hat\ap}$-multiplicity of $\mu\in \Delta^+\setminus \ihp$ equals the number of $\gamma\in\ihp$ such that  $\mu\in  \eus N(w_{\gamma,\hat\ap})$. Therefore, the required 
inclusion is equivalent to that 
\[
 \eus N(w_{\theta ,\hat\ap})\cap 
\bigl(\bigcup_{\gamma\in\ihp\setminus\{\theta\}} \eus N(w_{\gamma,\hat\ap})\bigr)=\varnothing .
\]
We have $ \eus N(w_{\theta ,\hat\ap})\subset \Delta_{\hat\ap}(1)$, see Lemma~\ref{lm:mu-min}.
Therefore, it suffices to prove that $\eus N(w_{\gamma,\hat\ap})\subset \Delta_{\hat\ap}(0)^+$ for $\gamma\ne\theta$.  First, consider $\gamma=\theta-\hat\ap$, the only root covered by $\theta$. As above,  we set
$\hat w=w_{\theta-\hat\ap,\hat\ap}$. Then $\hat w\sigma_{\hat\ap}=w_{\theta,\hat\ap}$ and 
\[
   \eus N(\hat w)=\sigma_{\hat\ap}\bigl( \eus N((w_{\theta,\hat\ap}) \setminus \{\hat\ap\}\bigr) .
\]
Here we have to prove that if $\nu\in  \eus N((w_{\theta,\hat\ap}) \setminus \{\hat\ap\}$, then $(\nu,\hat\ap)>0$
and thereby $\sigma_{\hat\ap}(\nu)=\nu-\hat\ap\in \Delta_{\hat\ap}(0)^+$. Let us exclude the possibility that
$(\nu, \hat\ap) \le 0$.

(a) Assume that $(\nu, \hat\ap)< 0$. Then $\hat\ap+\nu$ has the $\hat\ap$-height $2$, i.e., 
$\hat\ap+\nu=\theta$.
However, $\theta-\hat\ap\in\ihp$, since $\hat\ap\in \eus N(w_{\theta,\hat\ap})$, see Lemma~\ref{lm:N(w_theta,ap)}. A contradiction!

(b) Assume that  $(\nu, \hat\ap)=0$. Then $\hat w(\nu)=w_{\theta,\hat\ap}(\nu)$ is negative. Since
$\hat w$ is an involution (Prop.~\ref{prop:hat-w-invol}), $\nu\in \eus N(\hat w^{-1})\subset 
\eus N(w_{\theta,\hat\ap}^{\ -1})$. Hence one must have $(\nu,\hat\ap^\vee)=-1$. A contradiction!

This proves that $\eus N(\hat w)\subset \Delta_{\hat\ap}(0)^+$. For any other $\gamma\in \ihp$, we have
$\gamma\prec\theta-\hat\ap$. Then $w_{\gamma,\hat\ap}$ is a left factor of $\hat w$ and hence 
$\eus N(w_{\gamma,\hat\ap})$ is located below $\eus N(\hat w)$, see Proposition~\ref{lm:lezhit-nizhe} and 
Example~\ref{ex:left-factor}. Therefore, $\eus N(w_{\gamma,\hat\ap})$ lies in $\Delta_{\hat\ap}(0)^+$, too.

{\bf 2)} \ Let us prove that if $\mu\in \Delta_{\hat\ap}(0)^+$, then $\mu\not\in \ess (\eus F_{\hat\ap})$. In other
words, we have to prove that $n_{\hat\ap}(\mu)\ge 2$. We use the same approach as in the proof of 
Theorem~\ref{thm:ab-nilrad}. Consider
\[
  \Phi_\mu=\{\gamma\in \Delta^+\mid \gamma-\mu\in\Delta^+\}=\Phi_\mu(0)\cup\Phi_\mu(1)\cup\Phi_\mu(2) ,
\]
where $\Phi_\mu(i)$ denotes the subset in which $\gamma\in\Delta_{\hat\ap}(i)$. Note that 
$\Delta_{\hat\ap}(2)=\{\theta\}$ and $\theta-\mu$ is not a root. Hence $\Phi_\mu(2)=\varnothing$. Recall that
$\Delta_{\hat\ap}(0)$ is the disjoint union of the irreducible subsystems $\Delta_{\hat\ap}(0)_j, \ j\in J$, and
$h_j$ is the Coxeter number of $\Delta_{\hat\ap}(0)_j$. By
Lemma~\ref{lm:Phi-gamma}, $\#\Phi_\mu=h-1-\hot(\mu)$ and $\#\Phi_\mu(0)=h_j-1-\hot(\mu)$ if
$\mu\in \Delta_{\hat\ap}(0)_j$.  Hence $\#\Phi_\mu(1)=h-h_j$. However, 
$\Delta_{\hat\ap}(1)=\ihp\cup \eus N(w_{\theta,\hat\ap})$, and we only need to count the subset of $\Phi_\mu(1)$
corresponding to $\gamma\in \ihp$.

Given $\mu\in \Delta_{\hat\ap}(0)^+$, a sum $\mu+\mu'=\gamma\in \Delta_{\hat\ap}(1)$ is said to be {\it good} (resp. {\it bad}) if $\gamma\in \ihp$ (resp. $\gamma\in \eus N(w_{\theta,\hat\ap})$). Accordingly, 
$\Phi_\mu(1)=\Phi_\mu(1)^{good}\cup \Phi_\mu(1)^{bad}$ and $n_{\hat\ap}(\mu)=\#\Phi_\mu(1)^{good}$. 
If $\mu+\mu'$ is bad, then 
$\nu= \theta-(\mu+\mu')$ is a root and, moreover, $\nu\in \ihp$, see Lemma~\ref{lm:N(w_theta,ap)}. 
Since $\mu+\mu'+\nu=\theta$ and $\mu'\in
\Delta_{\hat\ap}(1)$, $\mu+\nu$ is also a root. Furthermore, since $\nu\in\ihp$, we have $\nu+\mu\in\ihp$, too.
Thus, starting with any bad sum $\mu+\mu'$, we obtain a good sum $\mu+\nu$. This yields an injection
$\Phi_\mu(1)^{bad}\hookrightarrow \Phi_\mu(1)^{good}$ and hence 
\[
   n_{\hat\ap}(\mu)=\#\Phi_\mu(1)^{good}\ge \frac{1}{2}(h-h_j) .
\]
It remains to prove that $h-h_j\ge 4$ for any irreducible subsystem $\Delta_{\hat\ap}(0)_j$ of 
$\Delta_{\hat\ap}(0)$. In other words, $\hot(\theta)-\hot(\mu)\ge 4$ for any $\mu\in\Delta_{\hat\ap}(0)^+$.
Indeed, since $[\theta:\hat\ap]=2$, a passage from $\theta$ to $\mu$ must contain two subtractions of
$\hat\ap$. Because $\theta-2\hat\ap\not\in\Delta$, such a passage should also include a subtraction of 
another $\ap\in\Pi$, necessarily adjacent to $\hat\ap$. Then one might assume that 
$\theta-2\hat\ap-\ap$ was a root
in $\Delta_{\hat\ap}(0)^+$. But $\theta-2\hat\ap-\ap$ is not a root in the simply-laced case! 
(Cf. the proof of Lemma~\ref{lm:mu-min}.) Thus, one needs 
at least $4$ subtractions.
See also the following Example. 

{\bf 3)} \ Here $\mathrm{defect}(\eus F_{\hat\ap})=\#\eus N(w_{\theta,\hat\ap})=\hot(\theta)-1=h-2$.
\end{proof}
\begin{ex}  \label{ex:theta-fund-h-hj}  We provide the differences $h-h_j$ for the irreducible subsystems of
$\Delta_{\hat\ap}(0)$ in all simply-laced cases with fundamental $\theta$.
\begin{enumerate}
  \item For $\GR{E}{6}$,  
$\Delta_{\hat\ap}(0)$ is of type $\GR{A}{5}$ and $h-h(\GR{A}{5})=12-6=6$;
  \item For $\GR{E}{7}$, 
$\Delta_{\hat\ap}(0)$ is of type $\GR{D}{6}$ and $h-h(\GR{D}{6})=18-10=8$;
  \item For $\GR{E}{8}$, 
$\Delta_{\hat\ap}(0)$ is of type $\GR{E}{7}$ and 
$h-h(\GR{E}{7})=30-18=12$;
\item For $\GR{D}{n}$ ($n\ge 4$), 
$\Delta_{\hat\ap}(0)$ is of type $\GR{A}{1}\times \GR{D}{n-2}$. Then  $h-h(\GR{D}{n-2})=(2n-2)-(2n-6)=4$ and $h-h(\GR{A}{1})=(2n-2)-2=2n-4\ge 4$.
\end{enumerate}
\end{ex}

\begin{rmk}
Unlike the case in which $I(\ap)_{max}$ corresponds to an abelian nilradical, see Theorem~\ref{thm:ab-nilrad}, the $\eus F_{\hat\ap}$-multiplicities
are not constant on the irreducible subsystems of $\Delta_{\hat\ap}(0)$. The reason is that different $\mu$ can participate in a different number of {\it good\/} and {\it bad\/} sums.
\end{rmk}

\section{Some speculations}  \label{sect:fin}

Our previous results on \MS\ of the form $\eus F_\ap$ ($\ap\in\Pi=\Pi_l$) describe $\ess(\eus F_\ap)$
for
\begin{itemize}
\item all simple roots for $\GR{A}{n}$;
\item $\ap_1,\ap_2=\hat\ap,\ap_{n-1},\ap_n$  for $\GR{D}{n}$ ($n\ge 4$); 
\item $\ap_1,\ap_5,\ap_6=\hat\ap$  for $\GR{E}{6}$;
\item $\ap_1,\ap_6=\hat\ap$  for $\GR{E}{7}$;
\item $\ap_1=\hat\ap$  for $\GR{E}{8}$.
\end{itemize}
The remaining cases concern the series  $\GR{D}{n}$ and  $\GR{E}{n}$, where $\theta$ is fundamental.
Our partial computations  
suggest the following

\begin{conj}    \label{conj:ess-fund}
If $\Delta$ is simply-laced and $\theta$ is fundamental, 
then\/ $\bigl(\ess(\eus F_{\ap})\setminus I(\ap)_{max}\bigr)\subset \gH$.
\end{conj}

Since $\gH=I(\ap)_{min}\sqcup \eus N(w_{\theta,\ap})$ (Lemma~~\ref{lm:N(w_theta,ap)}) and
$I(\ap)_{min}=I(\ap)_{max}\cap\gH$, it can also be restated as
$\ess(\eus F_{\ap})\subset I(\ap)_{max}\sqcup \eus N(w_{\theta,\ap})$. We also know that
$\#\{\gamma\in\Delta^+\mid (\gamma,\theta^\vee)=1\}=2\bigl( (\rho,\theta^\vee)-1\bigr)$ and
\[
\#I(\ap)_{max}=\#\eus N(w_{\theta,\ap})=(\rho,\theta^\vee)-1=h-2 . 
\]
Therefore, Conjecture~\ref{conj:ess-fund} would imply that $\#\ess(\eus F_{\ap})-\#\eus F_\ap\le h-2$.
Then a companion conjecture is
\begin{conj}    \label{conj:ess-fund2}
(1) \ If $\Delta$ is simply-laced and $\theta$ is fundamental, then $\mathrm{defect}(\eus F_\ap)\le h-2$;
\\  \indent
(2) \ Furthermore, $\mathrm{defect}(\eus F_\ap)=h-2$ if and only if $\ap=\hat\ap$ (i.e., $(\theta,\ap)\ne 0$).
\\  \indent
(3) \  $\mathrm{defect}(\eus F_\ap)=0$ if and only if $\ap\in\Pi$ is an endpoint of the Dynkin diagram that is different from $\hat\ap$.
\end{conj}

Note that if $[\theta:\ap]=1$, then $\mathrm{defect}(\eus F_\ap)=0$ (Theorem~\ref{thm:main}) and
$\ap$ is always an endpoint of the Dynkin diagram for $\GR{D}{n}$ and $\GR{E}{n}$.

\begin{ex}  \label{ex:E6}  
For $\Delta$ of type $\GR{E}{6}$,
straightforward computations show that both Conjectures are true,  
$\mathrm{defect}(\eus F_{\ap_2})=\mathrm{defect}(\eus F_{\ap_4})=2$, and 
$\mathrm{defect}(\eus F_{\ap_3})=4$.
\end{ex}
Another curious observation is that in all cases, where the explicit description is known,
$\ess(\eus F_\ap)$ appears to be the set of roots of a $\be$-{\sl stable subspace\/} of $\ut^+$. This should be compared with the fact all \MS\ in Example~\ref{ex:F4} do not have this property, and the reason is that $\ca$ there is not
an abelian ideal!

\end{document}